\documentclass[letter,10pt]{article}
\usepackage[utf8]{inputenc}
\usepackage[margin=1in]{geometry}
\usepackage{amsfonts}
\usepackage{amsthm}
\usepackage{fancyhdr}
\usepackage{comment}
\usepackage{times}
\usepackage{amsmath}
\usepackage{amsrefs}
\usepackage{mathtools}
\usepackage{esint}
\usepackage{changepage}
\usepackage{amssymb}
\usepackage{graphicx}
\usepackage{tikz}
\usepackage{tikz-cd}
\usepackage{titlesec}
\usepackage{xcolor}
\usetikzlibrary{arrows}
\usepackage[mathscr]{euscript}
\theoremstyle{plain}
\newtheorem{theorem}{Theorem}

\newtheorem{lemma}{Lemma}[section]

\theoremstyle{definition}

\theoremstyle{remark}

\newcommand{\R}{\mathbb{R}}
\newcommand{\C}{\mathbb{C}}
\newcommand{\Z}{\mathbb{Z}}
\newcommand{\N}{\mathbb{N}}
\newcommand{\T}{\mathbb{T}}

\newcommand{\X}{\mathbb{X}}
\newcommand{\Ind}{\mathbf{1}}

\newcommand{\supp}{\operatorname{supp}}

\newcommand{\setcomp}[1]{{#1}^{\mathsf{c}}}

\newcommand{\interior}[1]{%
	{\kern0pt#1}^{\circ}%
}

\let\epsilon\varepsilon

\titleformat{\section}{\centering\normalfont\Large\scshape}{\thesection.}{0.5em}{}
\titleformat{\subsection}{\centering\normalfont\large\scshape}{}{0em}{}

\title{\scshape{Existence and smoothness of extremizers \\ for convolution with \\ compactly supported measures}}
\author{James Tautges}
\date{}

\begin{document}

\maketitle
\vspace{-2em}

\begin{abstract}
    In this article, we establish various facts about extremizers for $L^p$-improving convolution operators $T\colon L^p \rightarrow L^q$ associated with compactly-supported probability measures on either $\R^d$ or $\T^d$. 
    If $\sigma$ has positive Fourier decay, we prove that extremizers exist and extremizing sequences are precompact modulo translation for all ``non-endpoint" $(p,q)$.
    These extremizers also satisfy an interesting positivity property and belong to $C^\infty_{loc} \cap L^\infty$.
\end{abstract}

\section{Introduction}

Let $d \geq 1$ and $\X \in \{\R^d, \T^d\}$.

For any linear operator $T\colon \mathcal{D}(\X) \rightarrow \mathcal{D}'(\X)$, define
\[
    R(T) = \{(\tfrac1p, \tfrac1q) \in [0,1]^2 : \|T\|_{p,q} < \infty\}.
\]

We call a pair of exponents $(p, q)$ \textit{valid} if $(\tfrac1p, \tfrac1q) \in R(T)$ and \textit{non-endpoint} if $(\tfrac1p, \tfrac1q) \in \interior{R(T)}$. 
Throughout, we will implicitly assume that $T$ is non-trivial.

If $(\tfrac1p, \tfrac1q) \in R(T)$, an \textit{extremizer} for $T\colon L^p(\X) \rightarrow L^q(\X)$ is a non-zero function $f \in L^p(\X)$ such that $\|Tf\|_q = \|T\|_{p,q}\|f\|_p$. 
An \textit{extremizing sequence} is a sequence $f_n \in L^p(\X)$ such that $\|Tf_n\|_q/\|f_n\|_p \rightarrow \|T\|_{p,q}$ as $n \rightarrow \infty$.

We say that a sequence $f_n \in L^p(\X)$ is \textit{normalized} if $\|f_n\|_p = 1$ for all $n$.

\begin{theorem}\label{T: existence}
    Let $\sigma$ be a compactly-supported probability measure such that $|\widehat\sigma(\xi)| \lesssim \langle \xi\rangle^{-\alpha}$ for some $\alpha > 0$ and
    \begin{equation}\label{E: conv op def}
        Tf(x) = \int_\X f(x-y) \, d\sigma(y).
    \end{equation}
    Assume that $1 < p < q < \infty$ is a valid, non-endpoint pair of exponents for $T$. 
    Then extremizing sequences for $T\colon L^p(\X) \rightarrow L^q(\X)$ are precompact (modulo translation when $\X = \R^d$). 
    In particular, extremizers exist.

    Furthermore, for all extremizers $f$ there exists a unimodular $\omega_0 \in \C$ such that
    \[
        \inf_{x \in W} \omega_0 f(x) > 0
    \]
    for all compact $W \subset \X$ and $f \in C^\infty_{loc}(\X) \cap L^\infty(\X)$.
\end{theorem}

Although the focus of this article is convolution operators, Theorem \ref{T: existence} bears the closest resemblance to results for adjoint Fourier extension operators associated with compact submanifolds.
The existence of extremizers for these operators was proven by Fanelli, Vega, and Visciglia for any compact submanifold and valid, non-endpoint pair of exponents $(2, q)$ such that $q \geq 2$ (\cite{FanelliEtAl2011}).
Flock and Stovall have also proven existence for adjoint Fourier restriction to the sphere for all valid, non-endpoint exponents $(p, q)$ such that $p < q$ (\cite{FlockStovall2022}).
Neither of these results address the properties of these extremizers.

A prototypical example of an $L^p$-improving operator is convolution with a measure $\sigma$ supported on a submanifold, which we call a generalized Radon transform. 
The study of $L^p$-improvement for generalized Radon transforms is extensive, involving techniques from combinatorics, sub-Riemannian geometry, and even ideas from algebraic geometry in the form of geometric invariant theory. 
We refer to Gressman for some of the furthest progress to date (\cite{Gressman2021}). 
At least in the case of convolution with surface measure on a compact hypersurface with positive Gaussian curvature, the endpoint $L^\frac{d+1}{d} \rightarrow L^{d+1}$ bound is known by the complex interpolation method of Littman (\cite{Littman1973}). 
In addition to their intrinsic interest as fundamental operators in harmonic analysis, such generalized Radon transforms have found applications in PDE, particularly in the study of propagation of singularities.

The existence of extremizers for generalized Radon transforms has been explored in several special cases---see the investigations of affine arclength measure on the moment curve by Biswas (\cite{Biswas2019}) and affine surface measure on the paraboloid by Christ (\cite{Christ2011}). 
These proofs of existence rely on the scaling symmetry of the operator to rule out extremizing sequences that converge weakly to zero either by ``spreading to infinity,'' or ``going up the spout.'' 
In our setting of compactly supported operators, we lose all scaling symmetries, but gain extra estimates at $L^1$ and $L^\infty$, allowing us to leverage interpolation to rule out the bad cases anyway. 
In \cite{Biswas2019} and \cite{Christ2011}, the convergence of extremizing sequences is deduced from properties of the \textit{quasi-extremizers}, or $L^p$-normalized functions $f$ for which $\|Tf\|_q \gtrsim \|f\|_p$ (\cite{Christ2011a}, \cite{Stovall2009}). 
Lemma \ref{L: profile decomposition} can be considered a rough replacement for that analysis, restricting quasi-extremizers to a small number of compact sets that differ by translation, in comparison with the many parameter paraballs in \cite{Biswas2019}, \cite{Christ2011a}, and \cite{Stovall2009}. 
In the case of surface measure on the sphere, quasi-extremals were characterized by Stovall (\cite{Stovall2009}), but, in comparison with the paraboloid or the moment curve, the lack of a product structure complicates the analysis of extremals used in \cite{Biswas2019} or \cite{Christ2011}.

Once existence is established, it is natural to ask about the identity of these extremizers.
In \cite{Christ2014}, Christ gave an explicit formula for the extremizers of convolution with the paraboloid.
These results also apply to the classical Radon transform given by integration over all hyperplanes passing through the origin. 
Flock extended the unique characterization to general $k$-plane transforms (\cite{Flock2016}).
Unfortunately, uniqueness or explicit formulas are too much to hope for in most situations, particularly at the level of generality of Theorem \ref{T: existence}.
For convolution with surface measure on the paraboloid with $(\tfrac{1}{p}, \tfrac{1}{q}) = (\tfrac{d}{d+1}, \tfrac{1}{d+1})$, Christ and Xue used a flexible inductive scheme to show that critical points of the functional $f \mapsto \|Tf\|_q/\|f\|_p$ (including but not exclusively extremizers) are smooth, with some extra decay (\cite{ChristXue2012}).

\subsection{Notation}

For a space of real measureable functions $X(\X)$, let
\[
	X_+(\X) = \{\phi \in X(\X) : \phi \geq 0\}.
\]

Let $C(\X)$ be the space of continuous functions on $\X$.

For all $\ell \geq 0$, let
\[
	C^\ell(\X) = \left\{ \phi \in C(\X) : \|\phi\|_{C^\ell(\X)} =
	\|\phi\|_{C^0(\X)} + \|\partial^\ell\phi\|_{C^0(\X)} < \infty\right\}
\]
and
\[
	C^\ell_{loc}(\X) = \left\{ \phi \in C(\X) : [\phi]_{\ell, K} = \|\phi\|_{C^0(K)}
    + \|\partial^\ell\phi\|_{C^0(K)} < \infty \;\forall\; K \subset \X \text{ compact}\right\}.
\]
Let
\[
	C^\infty(\X) = \bigcap_{\ell \geq 0} C^\ell(\X) \quad\text{and}\quad
	C^\infty_{loc}(\X) = \bigcap_{\ell \geq 0} C^\ell_{loc}(\X)
\]
with the induced seminorms. Let $\mathcal{D}(\X) = C^\infty_{cpct}(\X)$ be the subset of
$C^\infty(\X)$ with compact support and the same seminorms.

For all $s \geq 0$ and $1 < p < \infty$, let $W^{s,p}(\X)$ be the standard inhomogeneous Bessel
potential space (\cite[Section 2.1.2]{RunstSickel1996}).

The implicit constants contained in $\lesssim$ are allowed to depend upon $p$, $q$, $d$, and $\sigma$ unless otherwise noted.

\subsection{Acknowledgements}
The author would like to thank his advisor, Betsy Stovall, for all of her mathematical and practical
suggestions. This work was partially supported by NSF grants DMS-2246906 and DMS-2037851.

\section{Quasi-extremizer localization}

In this section, we prove that near-extremizers are concentrated on balls of uniform size when $\X = \R^d$. 
Let $c_d = \sqrt{d}$ so that there exists $0 < C < \infty$ such that
\[
    C^{-1} \leq \sum_{\alpha \in \Z^d} \Ind_{B(\alpha, c_d)}(x) \leq C
\]
for all $x \in \R^d$.
Without loss of generality, we may conjugate $T$ with a scaling of $\R^d$ and compose with a translation to reduce to the case $\supp \sigma \subset B(0,c_d)$.

\begin{lemma}\label{L: profile decomposition}
	Assume that $1 < p < q < \infty$. For all $\epsilon > 0$ sufficiently small, there exists a natural number $N \simeq \epsilon^{-\frac{pq}{q-p}}$ such
	that, for all $f \in L^p(\R^d)$ with $\|Tf\|_q \gtrsim \|T\|_{p,q}\|f\|_p$ and $\|f\|_p = 1$, there
	exists a sequence of non-negative functions $\phi_1, \dots, \phi_N$, a sequence
	$x_1, \dots, x_N \in \R^d$, and a remainder $r_N \in L^p(\R^d)$ with the following
	properties:
	\begin{enumerate}
		\item $\supp \phi_j \subset B(0,2c_d)$ for all $1 \leq j \leq N$;
		\item $\phi_j = \Ind_{B(0,2c_d)}\tau^{x_j}(f - \sum_{k=1}^{j-1} \tau^{-x_k}\phi_k)$ for all $1 \leq j \leq N$;
		\item $f = \sum_{j=1}^N \tau^{-x_j}\phi_j + r_N$; and
		\item $\|Tr_N\|_q < \epsilon$.
	\end{enumerate}
\end{lemma}
\begin{proof}
	Let $f \in L^p(\R^d)$ be any non-zero function such that $\|f\|_p \leq 1$ and $\|Tf\|_q > \epsilon$. 
    By the compact support of $\sigma$ and the fact that the collection $\{B(\alpha, 2c_d) : \alpha \in \Z^d\}$ covers each point in $\R^d$ at most $\lesssim 1$-fold,
	\begin{multline*}
		\epsilon^q < \int |Tf|^q \leq \sum_{\alpha \in \Z^d} \int_{B(\alpha, c_d)} |Tf|^q
		= \sum_{\alpha \in \Z^d} \int_{B(\alpha, c_d)} |T(f\Ind_{B(\alpha, 2c_d)})|^q 
		\leq \|T\|_{p,q}^q \sum_{\alpha \in \Z^d} \|f\Ind_{B(\alpha, 2c_d)}\|_p^q \\
		\leq \|T\|_{p,q}^q \sup_{\alpha \in \Z^d} \|f\Ind_{B(\alpha, 2c_d)}\|_p^{q-p} \sum_{\alpha \in \Z^d} \|f\Ind_{B(\alpha, 2c_d)}\|_p^p 
		\lesssim \|T\|_{p,q}^q \|f\|_p^p \sup_{\alpha \in \Z^d} \|f\Ind_{B(\alpha, 2c_d)}\|_p^{q-p} 
		\lesssim \sup_{\alpha \in \Z^d} \|f\Ind_{B(\alpha,2c_d)}\|_p^{q-p}.
	\end{multline*}
	Therefore, there exists $x_1 \in \Z^d$ such that
	\begin{equation}\label{E: bubble lower bound}
		\|f\Ind_{B(x_1, 2c_d)}\|_p^p \gtrsim \epsilon^\frac{pq}{q-p}.
	\end{equation}
	Let 
	$\phi_1 = \Ind_{B(0, 2c_d)}\tau^{x_1}f$. We can iterate this process inductively on
	the remainders $f - \sum_{k=1}^{j-1} \tau^{-x_k} \phi_k$, letting
	$\phi_j = \Ind_{B(0,2c_d)}\tau^{x_j}(f - \sum_{k=1}^{j-1}\tau^{-x_k} \phi_k)$, as
	long as $\|T(f-\sum_{k=1}^{j-1} \tau^{-x_k} \phi_k)\|_q > \epsilon$, resulting in
	sequences $x_1, \dots, x_j$ and $\phi_1, \dots, \phi_j$. By repeatedly applying \eqref{E:
	bubble lower bound}, there exists $C < \infty$ such that
	\begin{multline*}
		\|T(f - \sum_{k=1}^j \tau^{-x_k}\phi_k)\|_q^p \leq \|T\|_{p,q}^p \|f - \sum_{k=1}^j
		\tau^{-x_k} \phi_k\|_p^p \\
		\leq \|T\|_{p,q}^p (1 - C\epsilon^\frac{pq}{q-p})\|f -
		\sum_{k=1}^{j-1} \tau^{-x_k} \phi_k\|_p^p \leq \|T\|_{p,q}^p
		(1-Cj\epsilon^\frac{pq}{q-p}) \|f\|_p^p,
	\end{multline*}
	so for $\|f\|_p = 1$, the process ends after no more than $N$ steps, where
	$N \simeq \epsilon^{-\frac{pq}{q-p}}$.
\end{proof}

\begin{lemma}\label{L: spatial localization}
	For all sufficiently small $\eta > 0$, there exists $R > 0$ such that, for
	all $f \in L^p(\R^d)$ with $\|Tf\|_q > (1-\eta)\|T\|_{p,q}\|f\|_p$, there exists $x_0 \in \R^d$
	such that
	\[
		\|(\tau^{x_0}f) \Ind_{B(0,R)}\|_p = o_\eta(1) \|f\|_p
	\]
	uniformly in $f$ as $\eta \rightarrow 0$. For a given $f$, we can take the same $x_0$ for
	all $\eta$ for which $f$ satisfies the hypothesis.
\end{lemma}
\begin{proof}
	By the homogeneity of the inequality, we assume without loss of generality that $\|f\|_p = 1$.

	Applying Lemma \ref{L: profile decomposition} to $f$ at level $o_\eta(1)$ (a number chosen to be sufficiently small depending on $\eta$), 
    we have $N \in \N$, functions $\phi_j \in L^p(\R^d)$ such that $\supp \phi_j \subset B(0,2c_d)$ and $x_j \in \R^d$ for all $1 \leq j \leq N$, and $r_N \in L^p(\R^d)$ such that
	\[
		f = \sum_{j=1}^N \tau^{-x_j} \phi_j + r_N,
	\]
	as well as the other conclusions. We will replace $o_\eta(1)$ with $o(1)$ for the remainder of the proof. Since
	$\|Tr_N\|_q = o(1)$,
	\begin{equation}\label{E: remainder T small}
		\|T(f - r_N)\|_q = (1-o(1))\|T\|_{p,q}
	\end{equation}
	and hence
	\begin{equation}\label{E: remainder small}
		\|f - r_N\|_p = 1-o(1).
	\end{equation}

	Let $\mathcal{A} \subset \{1, \dots, N\}$ be the indices $j$ such that $x_1$ and $x_j$ belong to the same connected components of $\cup_{j=1}^N B(x_j, 4c_d)$, and let $\mathcal{B} = \{1, \dots, N\} \setminus \mathcal{A}$. 
	We define $A = \sum_{j \in \mathcal{A}}\tau^{-x_j}\phi_j$ and
	$B = \sum_{j \in \mathcal{B}} \tau^{-x_j}\phi_j$. By construction and the definition of
	convolution, $TA$ and $TB$ have disjoint support, so
	\begin{multline*}
		(1-o(1))^q\|T\|_{p,q}^q \leq \|T(f-r_N)\|_q^q = \|TA\|_q^q + \|TB\|_q^q \\
		\leq \|T\|_{p,q}^q \max\{\|A\|_p, \|B\|_p\}^{q-p} \|f - r_N\|_p^p \leq \|T\|_{p,q}^q \max\{\|A\|_p, \|B\|_p\}^{q-p}.
	\end{multline*}
	Since $\|\phi_1\|_p \gtrsim \|f\|_p$ by the proof of Lemma \ref{L: profile decomposition}, for all
	sufficiently small $\eta$ independent of $f$, the maximum is the first term, so
	\[
		\|A\|_p = 1 - o(1)
	\]
	and
	\begin{equation}\label{E: bad bound}
		\|B\|_p = o(1).
	\end{equation}

	By construction, $\supp A \subset B(0,4Nc_d)$, so
	\[
		\|(\tau^{x_1} f)\Ind_{\setcomp{B(0,4Nc_d)}}\|_p = \|(\tau^{x_1} A+
		\tau^{x_1}B+\tau^{x_1}r_N)\Ind_{\setcomp{B(0,4Nc_d)}}\|_p = \|(\tau^{x_1}B +
		\tau^{x_1} r_N)\Ind_{\setcomp{B(0,4Nc_d)}}\|_p = o(1)
	\]
	by \eqref{E: remainder small} and \eqref{E: bad bound}. Set $x_0 = x_1$.
\end{proof}

\section{Extremizers exist}\label{S: existence}

\begin{proof}[Proof of Theorem \ref{T: existence} existence]
	We will treat the case $\X = \T^d$ as a warmup for $\R^d$ later on. 
    Let $f_n \in L^p(\T^d)$ be a normalized, extremizing sequence. 
    By the Banach-Alaoglu theorem, there exists a subsequence in $n$ and $f \in L^p(\T^d)$ such that $f_n \rightharpoonup f$ as $n \rightarrow \infty$.
    The Fourier decay hypothesis implies that $T\colon L^2(\T^d) \rightarrow W^{\alpha, 2}(\T^d)$ is bounded for some $\alpha > 0$.
    Therefore, by interpolating with a bound $T\colon L^r(\T^d) \rightarrow L^s(\T^d)$ in a neighborhood of $(p,q)$, we obtain the estimate $T\colon L^p(\T^d) \rightarrow W^{\kappa, q}(\T^d)$ for some $\kappa > 0$.
    By the Sobolev embedding theorem, there exists a subsequence in $n$ and $F \in L^q(\T^d)$ such that $Tf_n \rightarrow F$ in $L^q(\T^d)$.

	It remains to prove that $\|f\|_p = 1$, as this will imply that $f_n \rightarrow f$ in $L^p(\T^d)$ by uniform convexity. 
    Indeed, since 
    \[
		\langle F, \phi\rangle = \lim_{n \rightarrow \infty} \langle Tf_n, \phi\rangle = \lim_{n \rightarrow \infty} \langle f_n, T^*\phi\rangle = \langle Tf, \phi\rangle
    \]
    for all test functions $\phi$, $F = Tf$. Therefore,
	\[
        \|T\|_{p,q} = \lim_{n \rightarrow \infty} \|Tf_n\|_q = \|F\|_q = \|Tf\|_q \leq \|T\|_{p,q} \|f\|_p.
    \]
	Dividing both sides by $\|T\|_{p,q}$ completes the proof.

	We now proceed to the slightly more difficult case $\X = \R^d$, although, as we will see, Lemma \ref{L: spatial localization} essentially reduces matters to the compact case. 
    As in the prequel, let $f_n \in L^p(\R^d)$ be a normalized, extremizing sequence for $T\colon L^p(\R^d) \rightarrow L^q(\R^d)$. 
    By Lemma \ref{L: spatial localization}, there exists a sequence $x_n$ such that
	\begin{equation}\label{E: existence 3}
		\lim_{R \rightarrow \infty} \limsup_{n \rightarrow \infty} \|\tau^{x_n} Tf_n - T(\tau^{x_n} f_n \Ind_{B(0,R)})\|_q = 0.
	\end{equation}
	For the sake of brevity, let
	\[
		f_n^R = \tau^{x_n} f_n \Ind_{B(0,R)}.
	\]
    By the Banach-Alaoglu theorem, there exists $f \in L^p(\R^d)$ such that $\tau^{x_n} f_n \rightharpoonup f$ after passing to a subsequence in $n$.
    As was the case for $\X = \T^d$ above, interpolation, Sobolev embedding and the uniformly bounded support of $Tf_n^R$ for fixed $R$ implies the existence of $F^R \in L^q(\R^d)$ and subsequences in $n$ along which $Tf_n^R \rightarrow F^R$ in $L^q(\R^d)$ for all $R$. 
    By diagonalization, we may take a single subsequence in $n$ that satisfies the claim for all $R$.

    Next, we prove that the sequence $F^R$ is Cauchy. Indeed, for all $R, R' \in \N$, there exists $n$ sufficiently large that
	\[
		\|F^R - F^{R'}\|_q \leq \|F^R - Tf_n^R\|_q + \|Tf_n^R - Tf_n^{R'}\|_q + \|F^{R'} - Tf_n^{R'}\|_q = o_n(1) + o_{\min\{R,R'\}}(1)
	\]
	by Lemma \eqref{L: spatial localization}. 
    Taking $n \rightarrow \infty$, we can see that the sequence $F^R$ is Cauchy, and therefore there exists $F \in L^q(\R^d)$ such that $F^R \rightarrow F$. 
    The key is that the conclusions of Lemma \ref{L: spatial localization} are uniform for sufficiently large $n$, so taking larger and larger $n$ doesn't alter the $o_{\min\{R,R'\}}(1)$ term.

    Let $\epsilon > 0$ and let $\phi \in C^\infty_{cpct}(\R^d)$ be such that $\|\phi\|_{q'} = 1$ and $\langle F, \phi\rangle > (1-\epsilon)\|F\|_q = (1-\epsilon)\|T\|_{p,q}$.
    Then by \eqref{E: existence 3},
    \begin{multline*}
        (1-\epsilon)\|T\|_{p,q} \leq \langle F, \phi\rangle = \lim_{R \rightarrow \infty} \langle F^R, \phi\rangle = \lim_{R \rightarrow \infty} \lim_{n \rightarrow \infty} \langle Tf_n^R, \phi\rangle = \lim_{n \rightarrow \infty} \langle \tau^{x_n} Tf_n, \phi\rangle \\ 
        = \lim_{n \rightarrow \infty} \langle \tau^{x_n} f_n, T^* \phi\rangle = \langle f, T^*\phi\rangle \leq \|T\|_{p,q} \|f\|_p \|\phi\|_{q'}.
    \end{multline*}
    Dividing by $\|T\|_{p,q}$ and letting $\epsilon$ tend to zero, $f_n \rightarrow f$ by convexity.
\end{proof}

\section{Extremizers have constant argument}\label{S: positivity}

We return to the general setup: $\X \in \{\R^d, \T^d\}$.

\begin{lemma}\label{L: positivity}
	For all extremizers $f$ there exists a unimodular constant $\omega_0 \in \C$ such that $\omega_0 f \geq 0$ almost everywhere.
\end{lemma}
\begin{proof}
	We identify $\C$ and $\R^2$ and let $\cdot$ be the vector dot product. 
    For every $x \in \X$ for which $T|f|(x)$ is defined and non-zero, define a measure on every Borel set $E \subset S^1 \subset \R^2$ by
	\[
		\mu_x(E) = T|\Ind_{\arg f \in E} f|(x)
	\]
	By definition,
	\[
		\mu_x(S^1) = T|f|(x) \quad \text{and} \quad \int_{S^1} \omega \; d\mu_x(\omega) = Tf(x).
	\]

	By the triangle inequality and Cauchy-Schwarz,
	\begin{multline}\label{E: existence proof 1}
		|Tf(x)|^2 = \left|\left(\int_{S^1} \omega \;d\mu_x(\omega)\right) \cdot \left(\int_{S^1} \nu \;d\mu_x(\nu)\right)\right| = \left|\int_{S^1 \times S^1} \omega \cdot \nu \; (d\mu_x \times d\mu_x)(\omega, \nu)\right| \\
		\leq \int_{S^1 \times S^1} |\omega\cdot \nu| \; (d\mu_x \times d\mu_x)(\omega, \nu) \leq \int_{S^1 \times S^1} |\omega| |\nu| \; (d\mu_x \times d\mu_x)(\omega, \nu) = (T|f|(x))^2.
	\end{multline}
	Recall that we assumed that $T|f|(x) > 0$ at the beginning of the proof, so we may assume that $\mu_x \not\equiv 0$. 
    Since $f$ is an extremizer and $|Tf(x)| \leq T|f|(x)$ by the triangle inequlaity, $|Tf(x)| = T|f|(x)$ for almost every $x \in \X$. 
    Otherwise, we would have $\|Tf\|_q < \|T|f|\|_q \leq \|T\|_{p,q}\|f\|_p$.
    For equality to hold for the Cauchy-Schwarz inequality in \eqref{E: existence proof 1}, we must have $\omega = \pm \nu$ for $\mu_x \times \mu_x$-almost every $(\omega, \nu)$. 
    In other words, $\supp \mu_x \times \mu_x \subset \{(\omega, \nu) : \omega = \pm \nu\}$. 
    Taking any $\omega_0 \in \supp \mu_x$, this means that $\supp \mu_x \subset \{\pm \omega_0\}$. 
    Equality in	the application of the triangle inequality in \eqref{E: existence proof 1} implies that either $\mu_x(\{\omega_0\}) = 0$ or $\mu_x(\{-\omega_0\}) = 0$. 
    Let $\omega_x$ be the point with positive measure. 
    By the definition of $\mu_x$, $\omega_x = \arg Tf(x)$.

	It remains to show that $\omega_x$ is constant in $x$.

	For $n \geq 2$ and $0 \leq j \leq n-1$, let
	\[
        E_{n,j} = \{ x \in \X : f(x) \neq 0 \text{ and } \arg f(x) \in [j \frac{2\pi}{n}, (j+1)\frac{2\pi}{n})\}
	\]
	and
	\[
        F_{n,j} = \{ x \in \X : T|f|(x) > 0 \text{ and } \omega_x \in [j \frac{2\pi}{n}, (j+1)\frac{2\pi}{n})\}.
	\]
	By the definition of $\mu_x$ and the preceeding argument, $Tf(x) = T(f \Ind_{E_{n,j}})(x)$ for
	almost every $x \in F_{n,j}$. Therefore, assuming without loss of generality that $\|f\|_p = 1$,
	\begin{multline*}
		\|T\|_{p,q}^q = \int |Tf|^q = \sum_{j=0}^{n-1} \int_{F_{n,j}} |Tf|^q = \sum_{j=0}^{n-1} \int_{F_{n,j}} |T(f\Ind_{E_{n,j}})|^q \leq \|T\|_{p,q}^q \sum_{j=0}^{n-1} \|f\Ind_{E_{n,j}}\|_p^q \\
		\leq \|T\|_{p,q}^q \max_{0 \leq j \leq n-1} \|f \Ind_{E_{n,j}}\|_p^{q-p} \sum_{j=0}^{n-1} \|f\Ind_{E_{n,j}}\|_p^p =\|T\|_{p,q}^q \max_{0 \leq j \leq n-1} \|f \Ind_{E_{n,j}}\|_p^{q-p}.
	\end{multline*}
	This implies that for every $n \geq 2$, there exists $0 \leq j_n \leq n-1$ such that $\|f \Ind_{E_{n,j_n}}\|_p = 1$, so $\supp f \subset E_{n,j_n}$. 
    Let $I_n = [j_n \frac{2\pi}{n}, (j_n + 1)\frac{2\pi}{n}]$. Since $I_{n+1} \subset I_n$ and $|I_n| \rightarrow 0$, the Heine-Borel theorem says that $\cap_n I_n = \{\omega_0\}$ for some $\omega_0 \in S^1$. 
    Therefore, $\omega_x = \omega_0$ is constant in $x$, and $\overline{\omega_0}f \geq 0$ almost everywhere.
\end{proof}

In light of Lemma \ref{L: positivity}, we may always assume that an extremizer $f$ is
non-negative.

All extremizers satisfy a recurrance relation called the Euler-Lagrange equation. Since $f$ is a
critical point of the functional
\[
	f \mapsto \frac{\|Tf\|_q}{\|f\|_p},
\]
we see that
\begin{equation*}
	\frac{d}{d\epsilon}{\Big|}_{\epsilon = 0} \frac{\|T(f + \epsilon g)\|_q}{\|f + \epsilon g\|_p} = 0
\end{equation*}
for all $g \in L^p(\X)$. Since
\[
	\frac{d}{d\epsilon}{\Big|}_{\epsilon = 0} \|T(f+\epsilon g)\|_q = \|Tf\|_q^{1-q} \int (Tf)^{q-1}
	\, Tg
\]
and
\[
	\frac{d}{d\epsilon}{\Big|}_{\epsilon = 0} \|f+\epsilon g\|_p = \|f\|_p^{1-p} \int f^{p-1}\, g,
\]
the quotient rule implies that
\[
	\|f\|_p \|Tf\|_q^{1-q} \int (Tf)^{q-1} \, Tg = \|Tf\|_q \|f\|_p^{1-p} \int f^{p-1}\, g.
\]
Rearranging,
\[
	\int g\, \left(\|f\|_p^p\, T^*(Tf)^{q-1} - \|Tf\|_q^q\, f^{p-1}\right) = 0.
\]
Letting $g$ range over all of $L^p(\X)$,
\begin{equation}
	\label{E: E-L}
	f = \|T\|_{p,q}^{-\frac{q}{p-1}} \left( T^* (T f)^{q-1}\right)^\frac{1}{p-1}.
\end{equation}
for all normalized, non-negative extremizers $f$.

\begin{proof}[Proof of Theorem \ref{T: existence} positivity]
    By Theorem \ref{T: existence} and Lemma \ref{L: positivity}, we may assume that $f$ is a normalized, non-negative extremizer.

    It suffices to assume that $q \geq 2$. 
    Indeed, if $p < q < 2$, then $(Tf)^{q-1}$ is an extremizer for $T^*\colon L^{q'}(\X) \rightarrow L^{p'}(\X)$ since $\|T^*\|_{q', p'} = \|T\|_{p,q}$.
    If we prove that
    \[
        \inf_{x \in B(0,R)} g > 0
    \]
    for all non-negative extremizers $g$ for $T^*\colon L^{q'}(\X) \rightarrow L^{p'}(\X)$ and $R > 0$, then the same holds for $f$ by \eqref{E: E-L}.
    Therefore, we assume that $q \geq 2$.

	Let $a > 1$. For all $x \in \X$ such that $Tf(x)$ is
	defined and all $t \geq 0$,
	\[
		t^a \geq Tf(x)^a + aTf(x)^{a-1}(t - Tf(x)).
	\]
	Setting $t = f(\cdot)$ and applying $T$, by the fact that $T\Ind = \Ind$,
	\begin{multline}\label{E: Jensen}
		T(f^a)(x) \geq T\left(Tf(x)^a + aTf(x)^{a-1}(f(\cdot) - Tf(x))\right) \\
		= Tf(x)^a \cdot T\Ind + aTf(x)^{a-1}(Tf(x) - Tf(x) \cdot T\Ind) = Tf(x)^a.
	\end{multline}

    Since $\sigma$ has positive Fourier decay, there exists $N \in \N$ such that $K = (\sigma * \widetilde\sigma)^{*N} \in C^0(\X)$, where the superscript $\cdot^{*N}$ represents the $N$th convolution power. 
    Furthermore, since $\sigma$ is a probability measure, $K(0) = 1$, and therefore there exists $r > 0$ such that $K(x) > \tfrac{1}{2}$ for all $x \in B(0,r)$. 
    This means that, for all $R > 0$, there exist $N \in \N$ and $c_R > 0$ such that
    \begin{equation}\label{E: TT* big}
        (\sigma * \widetilde\sigma)^{*N}(x) > c_R \text{ for all } x \in B(0,R).
    \end{equation}

    By \eqref{E: E-L} and \eqref{E: Jensen},
	\begin{equation}\label{E: E-L lower bound}
		f \gtrsim_N \big( (T^* T)^N f\big)^{(\frac{q-1}{p-1})^N}.
	\end{equation}
    Let $W \subset \X$ be compact. For sufficiently large $R > 0$,
    \[
        \inf_{x \in W} f(x) \geq \inf_{x \in B(0,R/2)} f(x) \geq c_R \inf_{x \in B(0,R/2)} \left( \int_{B(x, R)} f \right)^{(\frac{q-1}{p-1})^N} \gtrsim \left( \int_{B(0, R/2)} f \right)^{(\frac{q-1}{p-1})^N} > 0
    \]
    by \eqref{E: TT* big} and \eqref{E: E-L lower bound}.
\end{proof}

\section{Extremizers are smooth}

\begin{lemma}\label{L: upper bound}
    Let $f$ be a normalized, non-negative extremizer. 
    Then $f \in L^\infty(\X)$.
\end{lemma}
\begin{proof}
    We apply the same reduction as in the preceding proof to assume that $q \geq 2$.

    Let $1 \leq s \leq 2$. Since $|\widehat{\sigma}| \lesssim \langle \cdot \rangle^{-\alpha}$, there exist $N \in \N$ and $0 \leq \theta \leq 1$ such that
    \begin{equation}\label{E: upper 1}
        \|(\widetilde{\sigma} * \sigma)^{*N}\|_s 
        \leq \|(\widetilde{\sigma} * \sigma)^{*N}\|_1^\theta \|(\widetilde{\sigma} * \sigma)^{*N}\|_2^{1-\theta} 
        = \|\widehat{\sigma}\|_{4N}^{2N(1-\theta)}
        \lesssim \left( \int \langle \cdot \rangle^{-4N\alpha}\right)^{\frac{1-\theta}{2}}
        < \infty.
    \end{equation}
	
    We will now show that	
    \begin{equation}\label{E: higher L^p}
		f \in \bigcap_{s \in [p,\infty)} L^s(\X).
	\end{equation}
	by iterating \eqref{E: E-L}.
    Since $T\colon L^\infty(\X) \rightarrow L^\infty(\X)$ is bounded and $(p,q)$ is non-endpoint, we can use complex interpolation to construct a function $\mathcal{Q}\colon (1,\infty) \rightarrow (1,\infty)$ such that $T\colon L^s(\X) \rightarrow L^{\mathcal{Q}(s)}(\X)$ is bounded for all $s$ and $\mathcal{Q}(p) > q$.
    The adjoint, $T^*\colon L^s(\X) \rightarrow L^{\mathcal{Q}(s)}(\X)$, is also bounded as it is simply $T$ composed with a reflection.
    Furthermore, by interpolation, we can define $\mathcal{Q}$ so that $\mathcal{R}\colon \tfrac{1}{s} \mapsto \mathcal{Q}(s)^{-1}$ is convex. 
    By \eqref{E: E-L},
    \[
		\|f\|_{\mathcal{Q}\left(\frac{\mathcal{Q}(s)}{q-1}\right)(p-1)} 
        \simeq \|\left(T^*(Tf)^{q-1}\right)^\frac{1}{p-1}\|_{\mathcal{Q}\left(\frac{\mathcal{Q}(s)}{q-1}\right)(p-1)} 
        \lesssim \|(Tf)^{q-1}\|_\frac{\mathcal{Q}(s)}{q-1}
        \lesssim \|f\|_s^\frac{q-1}{p-1}
    \]
	for all $p \leq s < \infty$, so it suffices to prove that
	\[
		\lim_{n\rightarrow \infty} \mathcal{S}^n(\tfrac1p) = 0 \quad\text{where}\quad \mathcal{S}(t) = \frac{1}{p-1}\cdot \mathcal{R}\big(\mathcal{R}(t)\cdot(q-1)\big)
	\]
	for all $0 \leq t \leq \tfrac1p$. 
    Indeed, by direct computation, $\mathcal{S}(0) = 0$, $\mathcal{S}(\tfrac1p) < \tfrac1p$, and $\mathcal{S}$ is convex. 
    Therefore, the sequence $\mathcal{S}^n(\tfrac1p)$ decreases and converges to some $t_0 \in [0,\tfrac1p]$. 
    Since $\mathcal{S}$ is continuous, this implies that $\mathcal{S}(t_0) = t_0$, and hence $t_0 = 0$. 
    This proves \eqref{E: higher L^p}.

    We apply \eqref{E: Jensen} to \eqref{E: E-L} to show that
	\begin{equation}\label{E: continuity 1}
        f \lesssim_N \left( (T^*T)^N f^{\frac{(q-1)^N}{(p-1)^{N-1}}}\right)^\frac{1}{p-1}.
	\end{equation}	
    Since $f^{\frac{(q-1)^N}{(p-1)^{N-1}}} \in L^{\max\{2, p\}}(\X)$ by \eqref{E: higher L^p} and $(\widetilde{\sigma} * \sigma)^{*N} \in L^{\min\{2,p'\}} (\X)$ by \eqref{E: upper 1}, $(T^* T)^N f^{\frac{(q-1)^N}{(p-1)^{N-1}}} \in L^\infty(\X)$ by Young's inequality.
    By \eqref{E: continuity 1},	$f \in L^\infty(\X)$.
\end{proof}

\begin{lemma}\label{L: smoothing estimate}
	There exists $\kappa > 0$ such that for all $s \geq 0$ and $\psi_1 \in \mathcal{D}_+(\X)$, there
	exists $\psi_2 \in \mathcal{D}_+(\X)$ such that
	\[
		\|\psi_1 T\phi\|_{W^{s+\kappa, p}(\X)} \lesssim \|\psi_2 \phi\|_{W^{s,p}(\X)}
	\]
	and $T\phi|_{\supp \psi_1} \equiv T(\psi_2 \phi)|_{\supp \psi_1}$ for all $\phi \in W^{s,p}(\X)$.
    The implicit constant depends on $\psi_1$.
\end{lemma}
\begin{proof}
	Since $T\colon L^r(\X) \rightarrow L^r(\X)$ is bounded for all $1 \leq r \leq \infty$ and $T$ commutes with translations, we see that $T\colon W^{s,r}(\X) \rightarrow W^{s,r}(\X)$ is bounded for all $1 < r < \infty$ and $s \geq 0$. 
    Furthermore, by the Fourier decay assumption, there exists $\alpha > 0$ such that $T\colon W^{s,2}(\X) \rightarrow W^{s+\alpha, 2}(\X)$. 
    Therefore, we may interpolate to find $\kappa > 0$ such that $T\colon W^{s,p}(\X) \rightarrow W^{s+\kappa, p}(\X)$ for all $s \geq 0$.

	By the compact support of $\psi_1$ and $\sigma$, there exists $\psi_2 \in \mathcal{D}_+(\X)$ such that $T\phi|_{\supp \psi_1} \equiv T(\psi_2 \phi)|_{\supp \psi_1}$ for all $\phi \in L^p(\X)$. 
    Therefore, by the Kato-Ponce inequality (\cite{GulisashviliKon1996}*{Theorem 1.4}), we have that
	\begin{align*}
        \|\psi_1 T\phi\|_{W^{s+\kappa,p}(\X)} 
        &= \|\psi_1 T(\psi_2\phi)\|_{W^{s+\kappa,p}(\X)} \\
        &\lesssim \|\psi_1\|_{W^{s+\kappa,\infty}(\X)} \|T(\psi_2 \phi)\|_p + \| \psi_1 \|_{\infty} \| T(\psi_2 \phi) \|_{W^{s+\kappa, p}(\X)} \\
        &\lesssim_{\psi_1} \| \psi_2 \phi \|_{p} + \| \psi_2 \phi \|_{W^{s,p}(\X)} 
        \lesssim \| \psi_2 \phi \|_{W^{s,p}(\X)}.
	\end{align*}
\end{proof}

\begin{lemma}\label{L: power estimate}
	Let $f$ be a normalized, non-negative extremizer. For all $s > 0$, $a \neq 1$, and $\psi_1 \in \mathcal{D}_+(\X)$, there exists
	$\psi_2 \in \mathcal{D}_+(\X)$ such that
	\[
		\|\psi_1 f^a\|_{W^{s,p}(\X)} \lesssim \|\psi_2 f\|_{W^{s,p}(\X)} \left(1 + \|\psi_2
		f\|_\infty^{\max\{1,s\}}\right).
	\]
    The implicit constant depends on both $\psi_1$ and $f$.
\end{lemma}
\begin{proof}
    By Theorem \ref{T: existence},
    \[
        \inf_{x \in \supp \psi_1} f(x) > C > 0 \quad\text{and}\quad \|f\|_\infty < \infty.
    \]
    Therefore, there exist $\eta \in C^\infty(\R^{\geq 0} \rightarrow \R^{\geq 0})$ and $\psi_2 \in \mathcal{D}_+(\X)$ such that $\eta(0) = 0$, $\eta(t) = t^a$ for all $C \leq t \leq \|f\|_\infty$, and thus
	\begin{equation}
		\label{E: eta}
		f^a|_{\supp \psi_1} \equiv \eta\circ (\psi_2 f)|_{\supp \psi_1}.
	\end{equation}
    Note that $\eta$ depends on both $f$ and $\psi_1$, so implicit constants in this proof are allowed to depend on both functions.
    By \cite{GulisashviliKon1996}*{Theorem 1.4},
    \begin{align*}
		\|\psi_1 f^a\|_{W^{s,p}(\X)} 
        &= \|\psi_1(\eta \circ (\psi_2 f))\|_{W^{s,p}(\X)} \\
        &\lesssim \|\psi_1\|_\infty \|\eta \circ (\psi_2 f)\|_{W^{s,p}(\X)} + \|\eta \circ (\psi_2 f)\|_p \|\psi_1\|_{W^{s,\infty}(\X)} \\
        &\lesssim_{\psi_1} \|\eta \circ (\psi_2 f)\|_{W^{s,p}(\X)} 
    \end{align*}
	If $0 < s < 1$,
    \[
        \|\eta \circ (\psi_2 f)\|_{W^{s,p}(\X)}
        \lesssim \|\psi_2 f\|_{W^{s,p}(\X)}\|\eta' \circ (\psi_2 f)\|_\infty
        \lesssim \|\psi_2 f\|_{W^{s,p}(\X)}\left( 1 + \|\psi_2 f\|_\infty\right)
    \]
    by \cite[Section 5.3.6 Theorem 1]{RunstSickel1996}.
    If $s \geq 1$,
	\[
		\|\eta \circ (\psi_2 f)\|_{W^{s,p}(\X)} \lesssim \|\psi_2
		f\|_{W^{s,p}(\X)}\left( 1 + \|\psi_2 f\|_\infty^s \right)
	\]
	by \cite[Section 5.3.6 Theorem 2]{RunstSickel1996}.
    In either case, the lemma is proved.
\end{proof}

Define the Euler-Lagrange operator
\[
	S\phi = \|T\|_{p,q}^{-\frac{q}{p-1}} \left(T^* (T\phi)^{q-1}\right)^\frac{1}{p-1}
\]
for all $\phi \in L^p_+(\X)$.

\begin{proof}[Proof of Theorem \ref{T: existence} smoothness]
    Applying Theorem \ref{T: existence}, we let $f$ be a normalized, non-negative extremizer.

	Let $\Omega \subset \X$ be a bounded domain and let $\psi_1 \in \mathcal{D}_+(\X)$ be such that $\psi_1|_\Omega \equiv \Ind$. 
    By Lemma \ref{L: smoothing estimate} and Lemma \ref{L: power estimate} (if $p=2$ or $q=2$, we skip one application), there exist $\kappa > 0$ and $\psi_2 \in \mathcal{D}_+(\X)$ such that
	\[
		\|\psi_1 Sf\|_{W^{s+\kappa,p}(\X)} \lesssim_s \|\psi_2 f\|_{W^{s,p}(\X)}
	\]
	for all $s \geq 0$. 
    Iterating, we obtain
	\[
		\|\psi_1 Sf\|_{W^{s,p}(\X)} \lesssim_s \|f\|_p
	\]
    for all $s \geq 0$, and thus $\psi_1 f \in C^\infty_{loc}(\X)$ by \eqref{E: E-L} and the Sobolev embedding theorm. 
    Since $\psi_1 f|_\Omega \equiv f|_\Omega$, $f \in C^\infty_{loc}(\X)$.
\end{proof}

\bibliography{exts_for_sphere.bib}

\end{document}